\documentclass[11pt]{amsart}
\input{amssym.def}
\input{amssym}
\newtheorem{pro}{Proposition}[section]
\newtheorem{recursion}[pro]{Recursion}
\newtheorem{lem}[pro]{Lemma}

\newtheorem{theo}[pro]{Theorem}
\newtheorem{defi}[pro]{Definition}
\newtheorem{cor}[pro]{Corollary}
\newtheorem{remk}[pro]{Remark}
\newtheorem{assu}[pro]{Assumption}

\newcommand{\ep}{\varepsilon}
\newcommand{\om}{\omega}
\newcommand{\vp}{\varphi}
\newcommand{\la}{\lambda}

\newcommand{\lra}{\longrightarrow}
\newcommand{\lmt}{\longmapsto}

\newcommand{\nrm}[1]{\mbox{ $ \displaystyle \left\| {#1} \right\| $} }
\newcommand{\nri}[1]{\mbox{ $ \nrm{ {#1} }_{\infty} $} }
\newcommand{\nrE}[1]{\mbox{ $ \nrm{ {#1} }_{E} $} }

\newcommand{\fk}[1]{ \left( {#1} \right) }

\newcommand{\bk}[1]{ \left\{ {#1} \right\} }
\newcommand{\btr}[1]{\mbox{ $ \left| {#1} \right| $ }}

\newcommand{\re}{{\bf\Bbb R}}
\newcommand{\rep}{{\bf\Bbb R^+}}
\newcommand{\za}{{\bf\Bbb N}}

\newcommand{\jz}{{\bf\Bbb J}}
\newcommand{\ji}{{\bf\Bbb I}}

\newcommand{\intO}{\int_{0}^{\infty}}

\newcommand{\nres}[1]{e^{-\frac{#1}{\lambda}}}

\newcommand{\Funk}[5]{ \begin{array}{ccccc}
                       {#1} & : & {#2} & \lra & {#3} \\
                            &   & {#4} & \lmt & \displaystyle{#5} 
                       \end{array}                       }

\newcommand{\Jl}{J_{\lambda}}
\newcommand{\Jlw}{J_{\lambda}^{\om}}

\newcommand{\ul}{u_{\lambda}}
\newcommand{\up}[2]{u_{#1,#2}}

\newcommand{\hw}{h^{\om}}

\newcommand{\dtl}{\fk{\partial_t}_{\la}}
\newcommand{\Tlp}{T_{\la}\psi}
\newcommand{\Tlpn}{T_{\la}^n\psi}

\begin{document}

\title{Whole Line Solutions to Abstract Functional Differential Equations}
\author{Josef Kreulich\ Universit\"at Duisburg/Essen\\}
\date{}
\begin{abstract}
In the underlying study it is shown how the linear method of the Yosida-approximation of the derivative applies to solve possibly nonlinear and multivalued functional differential equations like:
\begin{eqnarray*}
u^{\prime}(t) &\in& A(t,u_t)u(t) +\om u(t), \ t \in \re
\end{eqnarray*}
Furthermore, in the case of finite and infinite delay we give an answer about  whether the solution is bounded, periodic, almost periodic, or some kind of almost automorphy.
\end{abstract}

\maketitle

\section{Introduction}
In the following we are going to prove the existence and stability for nonlinear functional differential equations. For $r>0$ let
\begin{eqnarray*}
\ji &\in & \bk{(-\infty,0], [-r,0]},\\
E&\in& \bk{BUC(\re^-,X),C([-r,0],X)},\\
Y&=& BUC(\re,X).
\end{eqnarray*}
For $u\in Y$ we define $u_t:=\bk{\ji\ni s\mapsto u(t+s)}\in E.$ Based on the previous definition we consider the functional differential equation:

\begin{equation} \label{gFDE_re}
\begin{array}{rcl} 
u^{\prime}(t) &\in & A(t,u_t)u(t)+ \om u(t): \ t \in \re. \\
\end{array}
\end{equation}

This setting extends the one by Kartsatos \cite{bd-Kartsatos}, who proved the existence of a bounded solution on $\re$ in the case of infinite delay, $A(t,\vp)$ single-valued, and $D(A(t,\vp))=D.$ Litcanu \cite{litcanu} proved the existence of a periodic solution in the case of infinite delay as well. In these studies, it is also shown that the generalized solutions under certain conditions are in fact strong solutions in case $X$ is reflexive, and $\vp $ Lipschitz with $\vp(0)$ in the generalized domain of the operators $A(t,\cdot).$ In the present study we show how the completely different method of Yosida approximation of the derivative applies to obtain existence of strong and mild solutions. In the general Banach space case, and in view of the corresponding non-autonomous Cauchy problem,
\begin{equation} \label{first-Cauchy-Eq}
\begin{array}{rcl} 
u^{\prime}(t) &\in & B(t)u(t)+ \om u(t): \ t \in \re, \\
\end{array}
\end{equation}
it s shown that the found solution is an integral solution to (\ref{first-Cauchy-Eq}), when $B(t)=A(t,u_t),$ which serves for general regularity results in Banach spaces with the Radon-Nikodym-Property (RNP). 
Moreover, $A(t,\vp)$ may be multivalued and $D (A(t,\vp ) )$ may depend on $t\in\re.$ 
Additionally, we show how the method applies to provide the existence of bounded, periodic, anti-periodic, almost periodic and almost automorphic solutions in the finite and infinte delay case. The applied method is taken from the study \cite{Kreulichevo}. 

\section{Main Assumptions} \label{Abstract-FDEs}

The main assumptions to solve the problem (\ref{gFDE_re}) on the operator $A$ are:
\begin{assu} \label{IVFDE_1_re}
The family $\bk{A(t,\vp):t \in \re, \vp \in E}$ are m-dissipative operators.
\end{assu}
\begin{assu}\label{IVFDE_2_og_re}
There exist bounded and uniformly continuous functions $h,k:\re \to X,$
 a constant $K_0 >0,$ and $L_1,L_2: \rep \lra \rep$ continuous and monotone non decreasing,
such that for $\la >0$ and $t_1,t_2\in \re $ we have
\begin{eqnarray*}
\lefteqn{\nrm{x_1-x_2}} \\
&\le& \nrm{x_1-x_2 -\la(y_1-y_2)}+\la \nrm{h(t_1)-h(t_2)}L_1(\nrm{x_2}) \\
&&+ \la\nrm{k(t_1)-k(t_2)}L_2(\nrm{\vp_2})+K_0\la\nrE{\vp_1-\vp_2}
\end{eqnarray*}
for all $[x_i,y_i] \in A(t_i,\vp_i),$  $i=1,2, \ \vp_i\in E.$
\end{assu}

\begin{assu} \label{IVFDE_2_re}
There exist bounded and Lipschitz continuous functions $g,h,k:\re \to X, $ a constant $K_0 >0,$ and $L_1,L_2: \rep \lra \rep$ continuous and monotone non decreasing, 
such that for $\la >0$ and $t_1,t_2\in \re$ we have
\begin{eqnarray*}
\lefteqn{\nrm{x_1-x_2}} \\
&\le& \nrm{x_1-x_2 -\la(y_1-y_2)}+\la \nrm{h(t_1)-h(t_2)}L_1(\nrm{x_2}) \\
&&+ \la \nrm{g(t_1)-g(t_2)}\nrm{y_2}+\la \nrm{k(t_1)-k(t_2)}L_2(\nrm{\vp_2})+K_0\la\nrE{\vp_1-\vp_2}
\end{eqnarray*}
for all $[x_i,y_i] \in A(t_i,\vp_i),$  $i=1,2, \ \vp_i\in E.$
\end{assu}
Throughout this study the Lipschitz constants for $g, h, k$ will be denoted by $L_g, L_h, L_k.$ The assumptions above are extensions to the assumptions given for the Cauchy-problem (\ref{first-Cauchy-Eq}). Setting $K_0=0,$ and $k=0$ we obtain the assumptions given in \cite{Kreulichevo}.

\begin{remk} \label{D-vp_eq_D}
We have 
$$ \hat{D_{\vp_1}} = \hat{D_{\vp_2}} \mbox { } \forall \vp_1,\vp_2\in E.$$ 
Moreover, if $x\in \hat{D}_{\vp_0}$ for some $\vp_0\in E$ and $F\subset E$ is bounded we find some $K>0,$ such that 
$$
\sup_{t\in\re, \vp\in F} \btr{A(t,\vp)x} \le K.
$$
In view of the previous observations we define $\hat{D}=\hat{D}_{\vp}.$ Consequently we have $\overline{D}_{\vp}=\overline{\hat{D}}=:\overline{D}.$
\end{remk}
\begin{proof}[Proof of Remark \ref{D-vp_eq_D}]
Due to Assumption \ref{IVFDE_2_re} we have for given $\vp_1,\vp_2 \in E$, $t,s \in \re,$ and $ x\in \hat{D}(A(t,\vp_1)$ 
\begin{eqnarray*}
\lefteqn{\btr{A(s,\vp_2)x}}\\
&\le \btr{A(t,\vp_1)x}+\nrm{h(s)-h(t)}L(\nrm{x})+\nrm{g(s)-g(t)}\btr{A(t,\vp_1)x}+K_0\nrm{\vp_1-\vp_2}_E.
\end{eqnarray*}
A similar inequality comes with the Assumption \ref{IVFDE_2_og_re}.
\end{proof}

As we consider $A(t,u_t)+\om I$ we need the perturbed control inequality  of Assumption       
\ref{IVFDE_2_og_re} and  \ref{IVFDE_2_re}. This is computed similar to \cite[pp. 1056-1057]{Kreulichevo} and leads with
$$\Funk{h^{\om}}{\jz}{\fk{ X\times X,\nrm{\cdot}_1}}{t}{(h(t),\btr{\om}g(t))}
$$ and $L_1^{\om}=L(t)+t$ in case Assumption \ref{IVFDE_2_re} to the modified inequality: 

\begin{eqnarray} \label{perturbed_control}
\lefteqn{\nrm{x_1-x_2}} \nonumber  \\
&\le& \nrm{x_1-x_2 -\la(y_1-y_2)}+\la \nrm{h^{\om}(t_1)-h^{\om}(t_2)}L_1^{\om}(\nrm{x_2}) \\
&&+ \la \nrm{g(t_1)-g(t_2)}\nrm{y_2}+ \nrm{k(t_1)-k(t_2)}L_2(\nrm{\vp_2}) \nonumber\\
&&+K_0\la\nrE{\vp_1-\vp_2} +\om\nrm{x_1-x_2}, \nonumber
\end{eqnarray}
and in the case of Assumption \ref{IVFDE_2_og_re} to
\begin{eqnarray} \label{perturbed_control_og}
\lefteqn{\nrm{x_1-x_2}} \nonumber  \\
&\le& \nrm{x_1-x_2 -\la(y_1-y_2)}+\la \nrm{h^{\om}(t_1)-h^{\om}(t_2)}L_1^{\om}(\nrm{x_2}) \\
&&+ \nrm{k(t_1)-k(t_2)}L_2(\nrm{\vp_2})+K_0\la\nrE{\vp_1-\vp_2} +\om\nrm{x_1-x_2}. \nonumber
\end{eqnarray}
Throughout this study we define for $t\in\re,$ $x\in X$ and $\vp\in E$,
\begin{eqnarray*}
J_{\la}^{\om}(t,\vp)x&:=&\fk{I-\la (A(t,\vp)+\om I)}^{-1}x\\
&=&\fk{I-\frac{\la}{1-\la\om}A(t,\vp)}^{-1}\fk{\frac{1}{1-\la\om}x} \\
&=:&J_{\frac{\la}{1-\la\om}}(t,\vp)\fk{\frac{1}{1-\la\om}x},
\end{eqnarray*}
for all $0<\la<\frac{1}{\om}.$

\section{Main Results}

To identify the found solution in a more general setting we define integral solutions similar to \cite{Kreulichevo}
\begin{equation} \label{whole-Cauchy-problem}
v^{\prime}(t)\in B(t)v(t)+\om v(t), \ t\in\re. 
\end{equation}
Therefore we view $B(t)$ as an operator independent of $\vp\in E$, i.e. $K_0=0$ and $k=0$ in the Assumption \ref{IVFDE_2_og_re}, or Assumption \ref{IVFDE_2_re}.

\begin{defi}
Assume that either the Assumption \ref{IVFDE_2_og_re} with  or Assumption \ref{IVFDE_2_re}  is satisfied for the family $\bk{B(t):t\in I}$ with $K_0=0$ and $k=0.$
A continuous function 
$u:\re \to X$ is called an integral solution on $\re$ to (\ref{whole-Cauchy-problem}), if
\begin{eqnarray*} \label{line-integral-sol-ineq}
\nrm{u(t)-x}-\nrm{u(r)-x}&\le &\int_r^t\fk{[y,u(\nu)-x]_{+}+\om\nrm{u(\nu)-x}}d\nu \\
&&+L^{\om}_1(\nrm{x})\int_r^t\nrm{\hw(\nu)-\hw(r)}d\nu +\nrm{y}\int_r^t\nrm{g(\nu)-g(r)}d\nu
\end{eqnarray*}
for all $-\infty < r\le t <\infty ,$ and $[x,y]\in B(r)+\om I.$
In the case of Assumption \ref{IVFDE_2_og_re} we assume $(g=0).$
\end{defi}

The proof of existence is split into two main steps, the initialization  of a recursion $n=1$  its step from $n\to n+1,$ for every small $\la,$ and finally the computation of the double limit $\lim_{n\to\infty}\lim_{\la\to 0}.$ To approximate the solution the method provided in \cite{Kreulichevo} is used.
 We start with defining the Yosida approximation of the derivative on the whole line.
 
 $$
(\partial_t)_{\la}u(t):=\frac{1}{\la}\fk{u(t)-\frac{1}{\la}\int_0^{\infty}\nres{\tau}u(t-\tau)d\tau}.
$$

By the above definition we are able to define the recursion for the approximations.

\begin{recursion} \label{recursion}
Let $\psi\in Y.$ 
\begin{description}
\item{$n=1$)} $u_{1,\la}$ is the solution to
\begin{equation} \label{recursion_start}
\begin{array}{rcll} 
\fk{\partial_t}_{\la} u_{1,\la}(t) &\in & A(t,(\psi)_t)u_{1,\la}(t) +\om u_{1,\la}(t) &:t\in \re
\end{array}
\end{equation}
\item{$n\to n+1$)} If $\bk{u_{n,\la}}_{\la>0} \subset Y$ is the solution to the n-th equation we define
$u_{n+1,\la}$ to be the solution to:
\begin{equation} \label{recursion_step} 
\begin{array}{rcll} 
\fk{\partial_t}_{\la} u_{n+1,\la}(t) &\in & A(t,(u_{n,\la})_t)u_{n+1,\la}(t) +\om u_{n+1,\la}(t)&:t\in \re
\end{array}
\end{equation}
\end{description}
\end{recursion}

We are ready to state the main result of this section on the existence of a solution to (\ref{gFDE_re}) for the finite and infinite delay case.
As we found a sequence of functions $\bk{u_n}_{n\in\za}$ it remains to prove their uniform convergence and the independence of the starting point $\psi$ of the recursion.

\begin{theo} \label{halfline-uniform-convergent}
Let Assumption \ref{IVFDE_1_re} and either Assumption \ref{IVFDE_2_og_re} with $K_0<-\om,$ or Assumption \ref{IVFDE_2_re} with $\max\bk{K_0,L_g}<-\om$ hold. Further, let $\psi\in Y.$ 
The double sequence defined in the Recursion \ref{recursion} is uniformly convergent on $\re,$ the limit is independent of the starting point, and is an integral solution to the problem (\ref{whole-Cauchy-problem}) with $B(t)=A(t,u_t)$ with adequate control functions $h_B,L_B,$ and $g.$
We call this limit the solution to (\ref{gFDE_re}) on $\re.$
\end{theo}

As a direct consequence we obtain a bounded solution similar to \cite{bd-Kartsatos} in the infinite delay case, and as well in the case of finite delay.
\begin{theo}\label{theo-bd-sol}
If Assumption \ref{IVFDE_1_re} and either Assumption \ref{IVFDE_2_og_re} with $K_0<-\om,$ or Assumption \ref{IVFDE_2_re} with $\max\bk{K_0,L_g}<-\om$ hold, then there exist a bounded and uniformly continuous solution of (\ref{gFDE_re}). 
\end{theo}

In order to find strong solutions, the notion of Lipschitz continuity comes into play, for which we provide the following theorem.

\begin{theo} \label{theo-lipschitz-continuity}
If $A$ satisfy either Assumption \ref{IVFDE_2_og_re} with $h,k$ Lipschitz and $K_0<-\om,$ or Assumption \ref{IVFDE_2_re} with $L_g$ the Lipschitz constant $L_g+K_0<-\om,$  then the solution $u$ found in Theorem \ref{halfline-uniform-convergent} is globally Lipschitz continuous. 
\end{theo}

To obtain strong solutions in the case of RNP spaces or more restrictive for reflexive Banach spaces, apply the proof of \cite[6.37]{Ito_Kappel} and Theorem \ref{halfline-uniform-convergent}.

Moreover the approximation applies to obtain periodicity analogous to the result in \cite{litcanu}. Additionally, it applies for the finite delay case.
\begin{theo}\label{theo-periodic-sol}
Let Assumption \ref{IVFDE_1_re} and either Assumption \ref{IVFDE_2_og_re} with $K_0<-\om,$ or Assumption \ref{IVFDE_2_re} with $\max\bk{K_0,L_g}<-\om$ hold. Further, let for some $T>0$
$$
A(t+T,\vp)=A(t,\vp)
$$
for all $t\in \re$ and $\vp\in E.$ Then there exist a $T$-periodic solution of (\ref{gFDE_re}).
\end{theo}

\begin{remk}
A similar result can be found in the case of anti-periodicity.
$$ Y=\bk{f\in BUC(\re,X): f(t+T)=-f(t)}, $$ if
$$
J_{\la}(t+T,\psi_ {t+T})f(t+T)=-J_{\la}(t,\psi_t)f(t)
$$
for all $\psi,f\in Y.$
\end{remk}

The previous observation applies to closed and translation invariant subspaces of $BUC(\re,X)$ and leads to the follow abstract version. The method applies to periodic, antiperiodic, almost periodic and almost automorphic solutions 

\begin{theo}\label{theo-Y-sol}
Let Assumption \ref{IVFDE_1_re} and either Assumption \ref{IVFDE_2_og_re} with $K_0<-\om,$ or Assumption \ref{IVFDE_2_re} with $\max\bk{K_0,L_g}<-\om$ hold. Further let $Y\subset BUC(\re,X)$ a closed and translation invariant subspace. If
$$
\bk{t\mapsto J_{\la}^{\om}(t,\psi_t)f(t)} \in Y
$$
for all $\psi,f\in Y,$ then the equation (\ref{gFDE_re}) has a solution $u\in Y.$
\end{theo}

In the case of almost periodicity the forthcoming lemma applies to verify the assumption of the previous theorem.
\begin{lem}\label{ap-inv-lem}
If $\psi,f\in AP(\re,X),$ and $\bk{t\mapsto J_{\la}^{\om}(t,\vp)x} \in AP(\re,X),$ for all $\vp\in E,$ and $x\in X,$ then
$$
\bk{t\mapsto J_{\la}^{\om}(t,\psi_t)f(t)} \in AP(\re,X).
$$
\end{lem}

\begin{remk}
To find almost automorphic solutions consider 
$$Y=\bk{f\in BUC(\re,X): f \mbox{ is almost automorphic }} $$
\end{remk}

\section{Existence on the whole line}

The idea is to apply the Banach Fixpoint Principle on $Y$ for a forthcoming iteration. We start with the following proposition.

\begin{pro} \label{T-la-defined}
Let $\om <0,$ $\psi \in Y$, and $A(t,\psi_t)$ satisfy Assumption \ref{IVFDE_1_re} and  either Assumption \ref{IVFDE_2_og_re}, or Assumption \ref{IVFDE_2_re} then $\bk{x \mapsto J^{\om}_{\la}(t,\psi_t)x}$ is a strict contraction with $q=\frac{1}{1-\la\om}$, and
$$
F(u)(t):=\Jl^{\om}(t,(\psi)_t)\fk{\frac{1}{\la}\int_0^{\infty}\nres{\tau}u(t-\tau)d\tau}.
$$
has a fixpoint $u_{1,\la}\in Y=BUC(\re,X).$ 
\end{pro}
We define
\begin{equation}
\Funk{T_{\la}}{Y}{Y}{\psi}{u_{1,\la}}
\end{equation}

\begin{lem}\label{iteration-possible_whole-line}
Let Assumption \ref{IVFDE_1_re}, and either Assumption \ref{IVFDE_2_og_re} with $\om <0$, or Assumption \ref{IVFDE_2_re} with $L_g<-\om$ hold. If $\psi \in Y$ is Lipschitz, then 
\begin{enumerate}
 \item[a)]$\bk{T_{\la}\psi}_{\la>0}$ is uniformly bounded on $\re.$
  \item[b)] $T_\la\psi$ is equi Lipschitz on $\re$ in the case of  Assumption \ref{IVFDE_2_re}
 \item[c)]$T_\la\psi$ is uniformly equicontinuous on $\re$ in the case of  Assumption \ref{IVFDE_2_og_re}
 \item[e)]there exists $u_1 \in Y$ s.t $\lim_{\la\to 0} T_{\la}\psi \to u_1$ uniformly on $\re.$ 
 \end{enumerate}
\end{lem}

\begin{proof}[Proof of Lemma \ref{iteration-possible_whole-line}]
Under the Assumption \ref{IVFDE_2_re}, and awith $\psi \in Y$  given, the functions
\begin{equation} \label{modified_control}
\begin{cases}
\Funk{\tilde{h}}{[0,T]}{(X\times X \times E,\nrm{\cdot}_1)}{t}{(h^{\om}(t),k(t),\psi_t),} \\ 
\mbox{and}  \\
L(\nrm{x}):=L_1^{\om}(\nrm{x})+L_2(\nri{\psi})+K_0,
\end{cases}
\end{equation}
are the modified control functions. This definition implies

\begin{eqnarray} \label{modified_control}
\lefteqn{\nrm{x_1-x_2}}\nonumber \\
&\le & \nrm{x_1-x_2-\la(y_1-y_2)} +\la\nrm{\tilde{h}(t_1)-\tilde{h}(t_2)}L(\nrm{x}) \\
&& \la \nrm{g(t_1)-g(t_2)}\nrm{y_2}, \nonumber
\end{eqnarray}
for $(x_i,y_i)\in A(t_i,\psi_{t_i}).$  Therefore $B(t)=A(t,\psi_t)$ fulfills the Assumption 2.3 of \cite{Kreulichevo}.
A similar inequality comes with Assumption \ref{IVFDE_2_og_re}. Hence, in both cases with $B(t):=A(t,\psi_t)$ for $t\in\jz$ we are in the situation of \cite[Thm. 2.11, Thm 2.13(1)]{Kreulichevo}. \\
The item a) follows by \cite[Prop. 4.1, p. 1076]{Kreulichevo}. \\
The item b) is a consequence of \cite[Lemma 4.3, p. 1077]{Kreulichevo}, item c) of \cite[Cor. 4.3, p. 1078]{Kreulichevo}. The claim e) comes with \cite[Lemma 4.6, p. 1079]{Kreulichevo} which concludes the proof.
\end{proof}

\begin{remk}\label{history_control}
Let $\vp_i \in E,$ $i=1,2,$ $z\in X$ and $t\in \re,$ then
\begin{equation}
\nrm{J_{\la}^{\om}(t,\vp_1)z-J_{\la}^{\om}(t,\vp_2)z}\le \frac{K\la}{1-\la\om}\nrE{\vp_1-\vp_2}.
\end{equation}
\end{remk}
\begin{proof}[Proof of Remark \ref{history_control}]
Apply the inqualities (\ref{perturbed_control}) or (\ref{perturbed_control_og}) with $A_{\la}(t,\vp_i)z \in A(t,\vp_i)J_{\la}(t,\vp_i)z.$
\end{proof}

\begin{lem} \label{ul-fixpoint}
Let $\om<0,$ Assumption \ref{IVFDE_1_re}, and either Assumption \ref{IVFDE_2_og_re} or \ref{IVFDE_2_re} hold.
Further, let $\psi\in Y,$ and $T_{\la}$  satisfy the equation
\begin{equation} \label{psi-la-approx-gFDE}
\begin{array}{rcll} 
\fk{\partial_t}_{\la} T_{\la}\psi(t)&\in & A(t,(\psi)_t)T_{\la}\psi(t)+\om T_{\la}\psi(t)&:t\in\re
\end{array}
\end{equation}
If $t\in\re,$ and $\psi,\phi\in Y$ we have,
\begin{eqnarray}\label{iterate_ineq}
\lefteqn{\nrm{T_{\la}\psi(t) -T_{\la}\phi(t)}}\nonumber \\
&\le& \frac{\la K_0}{1-\la\om}\nrm{\psi_t-\phi_t}_E 
+\frac{K_0}{1-\la\om}\int_0^{\infty}\exp\fk{\frac{\om}{1-\la\om}\tau}\nrm{\psi_{t-\tau}-\phi_{t-\tau}}_E d\tau 
\end{eqnarray}

If $I(t):=(-\infty,t]$ then
\begin{eqnarray} \label{iterate_ineq_2}
\lefteqn{\sup_{x\in I(t)}\nrm{T_{\la}\psi(x) -T_{\la}\phi(x)}}\nonumber \\
&\le&  \frac{\la K_0}{1-\la\om}\sup_{x\in I(t)}\nrm{\psi(x)-\phi(x)} \nonumber \\
&&+ \frac{K_0}{1-\la\om}\int_0^{\infty}\exp\fk{\frac{\om}{1-\la\om}\tau}\sup_{x\in I(t-\tau)}        \nrm{\psi(x)-\phi(x)} d\tau
\end{eqnarray}
\end{lem}

\begin{proof}[Proof of Lemma \ref{ul-fixpoint}]
We restrict the proof to Assumption \ref{IVFDE_2_re}.
 For the starting element $\psi,$ due to Proposition \ref{T-la-defined} the mapping $T_\la$ is well defined. For given $t\in\re$ we find,
\begin{eqnarray*}
\lefteqn{\nrm{T_{\la}\psi(t)-T_{\la}\phi(t)} =\nrm{\up{\la}{\psi}(t)-\up{\la}{\phi}(t)}}\\
&\le & \nrm{ \Jlw(t,\psi_t)\fk{\frac{1}{\la}\int_0^{\infty}\nres{\tau}\up{\la}{\psi}(t-\tau)d\tau}-  \Jlw(t,\phi_t)\fk{\frac{1}{\la}\int_0^{\infty}\nres{\tau}\up{\la}{\phi}(t-\tau)d\tau} } \\
&\le&\nrm{\Jlw(t,\psi_t)\fk{\frac{1}{\la}\int_0^{\infty}\nres{\tau}\up{\la}{\psi}(t-\tau)d\tau} - \Jlw(t,\phi_t)\fk{\frac{1}{\la}\int_0^{\infty}\nres{\tau}\up{\la}{\psi}(t-\tau)d\tau} } \\
&&+ \nrm{\Jlw(t,\phi_t)\fk{\frac{1}{\la}\int_0^{\infty}\nres{\tau}\up{\la}{\psi}(t-\tau)d\tau} 
- \Jlw(t,\phi_t)\fk{\frac{1}{\la}\int_0^{\infty}\nres{\tau}\up{\la}{\phi}(t-\tau)d\tau}}\\
&\le& \frac{1}{\la(1-\la\om)}\int_0^{\infty}\nres{\tau}
   \nrm{\up{\la}{\psi}(t-\tau)-\up{\la}{\phi}(t-\tau)}d\tau \\
&& + \frac{\la K_0}{1-\la\om} \nrm{\psi_t-\phi_t}_E.\\
\end{eqnarray*}
By  Proposition \ref{integral-ineq} we obtain for all $t\in \re$
\begin{eqnarray}  \label{t_greater_0}
\lefteqn{\nrm{\up{\la}{\psi}(t)-\up{\la}{\phi}(t)} } \nonumber \\
&\le& \frac{\la K_0}{1-\la\om}\nrm{\psi_t-\phi_t}_E 
+\frac{K_0}{1-\la\om}\int_0^{\infty}\exp\fk{\frac{\om}{1-\la\om}\tau}\nrm{\psi_{t-\tau}-\phi_{t-\tau}}_E d\tau.
\end{eqnarray}
As 
$$ 
\nrE{\psi_t} \le \sup_{x\in I(t)}\nrm{\psi(x)},
$$
we find
\begin{eqnarray}  \label{t_greater_0}
\lefteqn{\nrm{\up{\la}{\psi}(t)-\up{\la}{\phi}(t)} } \nonumber \\
&\le&  \frac{\la K_0}{1-\la\om}\sup_{x\in I(t)}\nrm{\psi(x)-\phi(x)}  \\
&&+ \frac{K_0}{1-\la\om}\int_0^{\infty}\exp\fk{\frac{\om}{1-\la\om}\tau}\sup_{x\in I(t-\tau)}        \nrm{\psi(x)-\phi(x)} d\tau. \nonumber
\end{eqnarray}
Using that $\bk{t\mapsto\sup_{x\in I(t)}\nrm{\psi(x)}}$ is non-decreasing the proof is finished.

\end{proof}

\begin{lem} \label{induction_step_1}
Let Assumption \ref{IVFDE_1_re} and either Assumption \ref{IVFDE_2_og_re} with $\om <0$ or Assumption \ref{IVFDE_2_re} with $L_g<-\om$ hold. Further, let $\bk{\psi_{\la}}_{\la>0}\subset Y$ and $\psi\in Y.$ 
If
$$
\lim_{\la\to 0} \psi_{\la}=\psi, \mbox{ in } Y,
$$
$\ul\in Y$ the solution to
\begin{equation} 
\begin{array}{rcll} 
\fk{\partial_t}_{\la} \ul(t) &\in & A(t,(\psi_{\la})_t)\ul(t)+\om\ul(t)&:t\in\re
\end{array}
\end{equation}
and $v_{\la}\in Y$ the solution to 
\begin{equation} 
\begin{array}{rcll} 
\fk{\partial_t}_{\la} v_{\la}(t) &\in & A(t,(\psi)_t)v_{\la}(t)+\om v_{\la}(t)&:t\in\re
\end{array}
\end{equation}
then 
$$
\lim_{\la\to 0}\nrm{\ul-v_{\la}}_Y=0.
$$
\end{lem}

\begin{proof}[Proof of Lemma \ref{induction_step_1}]
Apply inequality (\ref{iterate_ineq_2}) from Lemma \ref{ul-fixpoint}  with $\phi=\psi_{\la}$ and $\psi=\psi,$ and compute the $\sup_{t\in\re}.$
\end{proof}

\begin{cor} \label{induction_step_2_halfline}
Under the conditons of the previous lemma, $\bk{v_{\la}}_{\la>0},$ and $\bk{\ul}_{\la>0}$ are Cauchy in Y for $\la\to 0.$
\end{cor}

\begin{proof}[Proof of Corollary \ref{induction_step_2_halfline}]
Using a uniform continuous extention of $\psi$ on $\re$, and a mollifier we find that the Lipschitz functions on $\ji\cup\jz$ are dense in $Y.$ Consider $\psi^{\ep}$ Lipschitz and the equation 
for 
\begin{equation} 
\begin{array}{rcll} 
\dtl v_{\la}^{\ep}(t) &\in & A(t,(\psi^{\ep})_t)v_{\la}^{\ep}(t)+\om v_{\la}^{\ep}(t)&:t\in \re\\
\end{array}
\end{equation}

Applying Lemma \ref{iteration-possible_whole-line} we find $\bk{v_{\la}^{\ep}}_{\la>0}$ is convergent for every $\psi^{\ep}$ arbitrary close to $\psi.$  
The use of (\ref{iterate_ineq}) from Lemma \ref{ul-fixpoint} and the triangle inequality gives
\begin{eqnarray*}
\nrm{v_{\la}(t)-v_{\mu}(t)} &\le & 
\nrm{v_{\la}(t)-v_{\la}^{\ep}(t)} + \nrm{v_{\la}^{\ep}(t)-v_{\mu}^{\ep}(t)}+\nrm{v_{\mu}^{\ep}(t)-v_{\mu}(t)} \\
&\le& \frac{\la K_0}{1-\la\om}\nrm{\psi_t-\psi^{\ep}_t}_E 
+\frac{K_0}{1-\la\om}\int_0^{\infty}\exp\fk{\frac{\om}{1-\la\om}\tau}\nrm{\psi_{t-\tau}-\psi^{\ep}_{t-\tau}}_E d\tau \\
&& + \nrm{v_{\la}^{\ep}(t)-v_{\mu}^{\ep}(t)} \\
&& +\frac{\mu K_0}{1-\mu\om}\nrm{\psi_t-\psi^{\ep}_t}_E +\frac{K_0}{1-\mu\om}\int_0^{\infty}\exp\fk{\frac{\om}{1-\mu\om}\tau}\nrm{\psi_{t-\tau}-\psi^{\ep}_{t-\tau}}_E d\tau. \\
\end{eqnarray*}
Thus it remains compute the $\sup_{t\in \re}$ on both sides of the inequality.
\end{proof}

\begin{lem} \label{iteration_convergent_halfline}
Let Assumption \ref{IVFDE_1_re} and either Assumption \ref{IVFDE_2_og_re} with $\om <0$ or Assumption \ref{IVFDE_2_re} with $L_g<-\om$ hold. Further, let  $\psi\in Y.$ 
If $u_{1,\la}$ is the solution to (\ref{recursion_start}) 
then 
$$
\lim_{\la \to 0}u_{1,\la}=u_1\in Y.
$$
If $\bk{u_{n,\la}}_{\la>0} \subset Y$ is the solution to the n-th step with 
$$\lim_{\la\to 0}u_{n,\la}=u_n \in Y,$$
and $u_{n+1,\la}$ is the solution to (\ref{recursion_step}),
then 
$$
\lim_{\la \to 0}u_{n+1,\la}=u_{n+1}\in Y.
$$
\end{lem}

\begin{proof}[Proof of Lemma \ref{iteration_convergent_halfline}] 
Apply Lemma \ref{iteration-possible_whole-line} and Corollary \ref{induction_step_2_halfline} to the start of the induction, the induction step follows by Lemma \ref{induction_step_1} and Corollary \ref{induction_step_2_halfline}.

\end{proof}

\begin{proof}[Proof of Theorem \ref{halfline-uniform-convergent}]
Let $t\in\re$, and $n\ge 1$ we have by 
Lemma \ref{ul-fixpoint} and inequality (\ref{iterate_ineq_2}), that
\begin{eqnarray*} 
\lefteqn{\sup_{x\in I(t)}\nrm{T^{n+1}_{\la}\psi(x) -T^{n}_{\la}\phi(x)}}\nonumber \\
&\le&  \frac{\la K_0}{1-\la\om}\sup_{x\in I(t)}\nrm{T^n_{\la}\psi(x)-T^{n-1}_{\la}\phi(x)} \nonumber \\
&&+ \frac{K_0}{1-\la\om}\int_0^{\infty}\exp\fk{\frac{\om}{1-\la\om}\tau}\sup_{x\in I(t-\tau)}        \nrm{T^n_{\la}\psi(x)-T^{n-1}_{\la}\phi(x)} d\tau
\end{eqnarray*}
As we may pass to $\la\to 0$, we obtain,
\begin{eqnarray*}
\sup_{x\in I(t)}\nrm{u_{n+1}(x)-u_n(x)} &\le&K_0 \int_0^{\infty}\exp\fk{\om \tau}\sup_{x\in I(t-\tau)}\nrm{u_n(x)-u_{n-1}(x)}d\tau \\
&\le& \frac{K_0}{\btr{\om}}\sup_{x\in I(t)}\nrm{u_n(x)-u_{n-1}(x)}.
\end{eqnarray*}
Consequently, when computing the $\sup_{t\in\re}$ we have
$$
\nrm{u_{n+1}-u_n}_Y \le \fk{\frac{K}{-\om}}^n \nrm{u_1-\psi}_Y
$$
and $\bk{u_n}_{n\in\za}$ becomes Cauchy in $Y.$  Let $\psi,\phi$ starting points of the iteration, then for 
$$
\Funk{S_{\la}}{BUC(\re,X)}{BUC(\re,X)}{f}{\bk{t\mapsto \frac{K_0}{1-\la\om}\int_0^{\infty}\exp\fk{\frac{\om\tau}{1-\la\om}}f(t-\tau)d\tau},}
$$
$I(t)=(-\infty,t],$ and 
$$
f^{\la}_n(t):=\sup_{x\in I(t)}\nrm{T^n_{\la}\psi(x)-T^n_{\la}\phi(x)} \ n\in\za_0
$$
by Lemma \ref{ul-fixpoint} we have,
\begin{eqnarray*}
f_{n+1}^{\la}(t)&\le&\frac{K_0\la}{1-\la\om}f_n^{\la}(t)+\fk{S_{\la}f_n^{\la}}(t).
\end{eqnarray*}

Iterating the inequality gives
\begin{eqnarray*}
\nri{f_{n+1}^{\la}} &\le& \nrm{\fk{\frac{K_0\la}{1-\la\om}I+S_{\la}}^{n+1}}\nri{\psi-\phi} \\
& \le & \fk{\frac{K_0\la}{1-\la\om}+\frac{K_0}{-\om}}^{n+1} \nri{\psi-\phi}\\
\end{eqnarray*}
we have for $\frac{K_0}{-\om} <1 $ and small $\la,$ that 
$
\fk{\frac{K_0\la}{1-\la\om}+\frac{K_0}{-\om}}^{n}$ is summable, therefore a null sequence, which proves the independence of the starting point, when passing to $\la\to 0$ and $n\to\infty.$

It remains to  prove that $u$ is an integral solution. We start with Assumption \ref{IVFDE_2_re}. For given $n\in\za$ we have for $q$ small that $\bk{T_{\la}^n\psi}_{0<\la\le q}$ is equi-Lipschitz by Lemma \ref{iteration-possible_whole-line}. Consequently, $u_n$ is Lipschitz as well.
Hence, with the modified control functions
$$
\Funk{\tilde{h}}{[0,T]}{(X\times X \times E,\nrm{\cdot}_1)}{t}{(h^{\om}(t),k(t),(u_n)_t)}
$$
and 
$$
L(\nrm{x}):=L_1^{\om}(\nrm{x})+L_2(\nri{u_n})+K_0,
$$ 
we obtain, that $B_n(t):=A(t,u^n_t)$ is an operator satisfying the Assumption 2.1, and 2.3 in \cite{Kreulichevo}. As the approximation 
\begin{equation}
\dtl v^n_{\la}(t)\in B_n(t)v^n_{\la}(t) +\om v^n_{\la}(t), t\in \re 
\end{equation}
tends to the integral solution \cite[2.11(2)]{Kreulichevo}, when passing to $\la\to 0.$.  Lemma \ref{induction_step_1} gives $u_n=v_n,$ and we conclude
\begin{eqnarray*}
\lefteqn{\nrm{u_n(t)-x_n}-\nrm{u_n(r)-x_n}\le \int_r^t\fk{[y_n,u_n(\nu)-x_n]_{+}+\om\nrm{u_n(\nu)-x_n}}d\nu} \\
&&+L(\nrm{x_n})\int_r^t\nrm{\tilde{h}(\nu)-\tilde{h}(r)}d\nu +\nrm{y_n}\int_r^t\nrm{g(\nu)-g(r)}d\nu
\end{eqnarray*}
for all $-\infty < r\le t <\infty ,$ and $[x_n,y_n]\in B_n(r)+\om I.$
Lemma \ref{history_control}, Assumption \ref{IVFDE_2_re} together with \cite[Theorem 10.5.]{Ito_Kappel} and \cite[(10.6)]{Ito_Kappel} lead, for given $[x,y]\in A(t,u_t)$ to $[x_n,y_n]\in A(t,(u_n)_t)$  such that $\lim_{n\to\infty}[x_n,y_n]=[x,y].$ Passing to $n\to \infty$ we showed that $u$ is an integral soltion in the case of Assumption \ref{IVFDE_2_re} with the control function $\tilde{h}, L,$ and $g.$ Note that similar arguments apply with uniform continuity in the case of Assumption \ref{IVFDE_2_og_re}, with \cite[Thm. 2.11 (1)]{Kreulichevo} instead of \cite[2.11(2)]{Kreulichevo}.
\end{proof}

\begin{proof}[Proof of Theorem \ref{theo-lipschitz-continuity}]
We restrict to the case of Assumption \ref{IVFDE_2_re}. and by the recursion \ref{recursion}, we have 
$$
\dtl T_{\la}\psi(t)\in A(t,\psi_t)\Tlp(t)+\om T_{\la}\psi(t).
$$
This leads for $t:=t-s$ and $t=t$ by the Assumption \ref{IVFDE_2_re} to:
\begin{eqnarray*}
\lefteqn{\nrm{\Tlp(t-s)-\Tlp(t)} \le \nrm{\Tlp(t-s)-\Tlp(t)-\la\fk{\dtl\Tlp(t-s)-\dtl\Tlp(t)}}} \\
&&+\frac{\la}{(1-\la\om)} \fk{ \nrm{h(t_1)-h(t_2)}L_1(\nrm{\Tlp(t)}) + \nrm{g(t_1)-g(t_2)}\nrm{\dtl\Tlp(t)}} \\
&&+ \frac{\la}{1-\la\om}\fk{\nrm{k(t_1)-k(t_2)}L_2(\nrE{\psi_t})+K_0\nrE{\psi_{t-s}-\psi_t} } \\
&\le& \frac{1}{\la(1-\la\om)}\int_0^{\infty}\exp\fk{\frac{-\tau}{\la}}\nrm{\Tlp(t-s-\tau)-\Tlp(t-\tau)}d\tau\\
&&+\frac{\la}{1-\la\om} \fk{ \nrm{h(t_1)-h(t_2)}L_1(\nrm{\Tlp(t)}) + \nrm{g(t_1)-g(t_2)}\nrm{\dtl\Tlp(t)}} \\
&&+ \frac{\la}{1-\la\om}\fk{\nrm{k(t_1)-k(t_2)}L_2(\nrE{\psi_t})+K_0\nrE{\psi_{t-s}-\psi_t} }
\end{eqnarray*}
Defining $I(t):=(-\infty,t]$
\begin{eqnarray*}
K_{\la,n}(t)&:=&\sup_{s>0}\frac{1}{s}\nrm{\Tlpn(t-s)-\Tlpn(t)} \\
K_{\la,n}^{\sup}(t)&:=&\sup_{r \in I(t)}K_{\la,n}(r) \\
S_{\la,n}(t)&=& \sup_{s>0}\frac{1}{s}\nrE{\Tlpn(t-s)-\Tlpn(t)}
\end{eqnarray*}
we have
$$
S_{\la,n}(t) \le K_{\la,n}^{\sup}(t),
$$
Additionally
$$
\nrm{\dtl\Tlpn(t)}\le K_{ \la,n}(t).
$$
This gives with boundedness of $\Tlpn$, the Lipschitz assumptions on $h, g,k,$ and some adequate contant K,
\begin{eqnarray*}
K_{\la,1}(t)&\le& \frac{1}{\la(1-\la\om)}\int_0^{\infty}\exp\fk{\frac{-\tau}{\la}}K_{\la,1}(\tau)d\tau \\
&&+\frac{\la}{1-\la\om}K\fk{L_h+L_k} +L_g K_{\la,1}(t) +\frac{K_0\la}{1-\la\om} S_{\la,0}(t). \\
\end{eqnarray*}
Consequently,
\begin{eqnarray*}
K_{\la,1}(t)
&\le& \frac{1}{\la(1-\la(\om+L_g))}\int_0^{\infty}\exp\fk{\frac{-\tau}{\la}}K_{\la,1}^{\sup}(\tau)d\tau \\
&&+\frac{\la}{1-\la(\om+L_g)}K\fk{L_h+L_k} +\frac{K_0\la}{1-\la(\om+L_g)} K_{\la,0}^{\sup}(t)
\end{eqnarray*}
As the right hand side is monotone increasing in $t$ we find
\begin{eqnarray*}
K_{\la,1}^{\sup}(t)
&\le& \frac{1}{\la(1-\la(\om+L_g))}\int_0^{\infty}\exp\fk{\frac{-\tau}{\la}}K_{\la,1}^{\sup}(\tau)d\tau \\
&&+\frac{\la}{1-\la(\om+L_g)}K\fk{L_h+L_k} +\frac{K_0\la}{1-\la(\om+L_g)} K_{\la,0}^{\sup}(t)
\end{eqnarray*}
Applying Proposition \ref{integral-ineq} we conclude with some adequate constant $K_2,$
\begin{eqnarray*}
\lefteqn{K_{\la,1}^{\sup}(t)} \\
&\le& 
\frac{\la}{1-\la(\om+L_g)}K\fk{L_h+L_k} +\frac{K_0\la}{1-\la(\om+L_g)} K_{\la,0}^{\sup}(t)\\
&&+\frac{1}{\la(1-\la(\om+L_g))} \intO \exp\fk{\frac{\om+L_g}{1-\la(\om+\L_g)}\tau}\frac{\la}{1-\la(\om+L_g)}K\fk{L_h+L_k}d\tau \\
&& +\frac{K_0}{(1-\la(\om+L_g))^2} \intO \exp\fk{\frac{\om+L_g}{1-\la(\om+\L_g)}\tau} K_{\la,0}^{\sup}(t-\tau)d\tau \\
&\le& K_2+\la K_2 K_{\la,0}^{\sup}(t))+ \frac{K_0}{(1-\la(\om+L_g))^2}\intO \exp\fk{\frac{\om+L_g}{1-\la(\om+\L_g)}\tau} K_{\la,0}^{\sup}(t-\tau)d\tau. \\
\end{eqnarray*}
By repeating the steps for every $n$ we find
\begin{eqnarray*}
\lefteqn{K_{\la,n+1}^{\sup}(t)} \\
&\le& K_2+\la K_2 K_{\la,n}^{\sup}(t))+ \frac{K_0}{(1-\la(\om+L_g))^2}\intO \exp\fk{\frac{\om+L_g}{1-\la(\om+\L_g)}\tau} K_{\la,n}^{\sup}(t-\tau)d\tau, 
\end{eqnarray*}
Defining
$$
\Funk{S_{\la}}{BUC(\re,X)}{BUC(\re,X)}{f}{\bk{t\mapsto \frac{K_0}{(1-\la(\om+L_g))^2}\intO\exp\fk{\frac{(\om+L_g)\tau}{1-\la(\om+L_g)}}f(t-\tau)d\tau}}
$$
for $L_g+K_0<-\om$ and small $\la$ we have that the spectral radius 
$r\fk{K_2\la I+S_{\la}} \le q<1 $   which leads to the uniform boundedness of $K_{ \la,n}^{\sup}.$

\end{proof}

Thus we are in the situation to apply the uniform convergence of the approximation $u_{\la,n}.$

\begin{proof}[Proof of Theorem \ref{theo-bd-sol}]
The fixpoint mapping of Prop. \ref{T-la-defined} leaves 
$ Y=BUC(\re,X)$
invariant. Thus, we find $u$ as a uniform limit of bounded and uniformly continuous functions.
\end{proof}

\begin{proof}[Proof of Theorem \ref{theo-periodic-sol}]
As the fixpoint mapping of Prop. \ref{T-la-defined} leaves 
$$ Y=\bk{f\in BUC(\re,X): f(t+T)=f(t) \mbox{ for all } t\in \re } $$
invariant, we find $u$ as a uniform limit of $T$-periodic functions.
\end{proof}

\begin{proof}[Proof of Lemma \ref{ap-inv-lem}]
For $x,y\in X$ and $\vp,\phi\in E$ by the triangle inequality
\begin{eqnarray*}
\nrm{J_{\la}^{\om}(t,\vp)x-J_{\la}^{\om}(t,\phi)y}&\le&\frac{\la K_0}{1-\la\om}\nrE{\vp-\phi}+\nrm{x-y}.
\end{eqnarray*}
For given   $g\in AP(\re,X)$ we have $\bk{t\mapsto g_t}\in AP(\re,AP(\re,X)),$ which yields $\bk{t\mapsto g_t}\in AP(\re,E)).$ Hence \cite[Chapter VII,Lemma 4.1]{Dal-Krein} applies.
\end{proof}

\section{Appendix}
\begin{pro} \label{integral-ineq}
The solution to the integral equation 
\begin{equation} \label{infinite-integral-inequality}
u(t) = f(t) + \alpha\intO\exp(- \beta \tau) u(t-\tau) d\tau, 
\end{equation}
 for 
$ 0 < \alpha < \beta,  $ 
is given by 
$$ u(t) = (Rf)(t):= f(t) + \alpha \intO\exp(-(\beta-\alpha)\tau) f(t-\tau)d\tau. $$ 
Note that the resolvent is positive.
\end{pro}

\end{document}